\documentclass[a4paper]{amsart}
\usepackage{amsmath,amsthm,amssymb,latexsym,epic,bbm,comment,mathrsfs}
\usepackage{graphicx,enumerate,stmaryrd, color}
\usepackage[all,2cell]{xy}
\xyoption{2cell}

\newtheorem{theorem}{Theorem}
\newtheorem{lemma}[theorem]{Lemma}
\newtheorem{corollary}[theorem]{Corollary}
\newtheorem{proposition}[theorem]{Proposition}

\usepackage[all]{xy}
\usepackage[active]{srcltx}
\usepackage[parfill]{parskip}
\usepackage{enumerate}

\newcommand{\tto}{\twoheadrightarrow}

\begin{document}

\title[Category $\mathcal{O}$ for Takiff $\mathfrak{sl}_2$]
{Category $\mathcal{O}$ for Takiff $\mathfrak{sl}_2$}

\author[V.~Mazorchuk and C.~S{\"o}derberg]
{Volodymyr Mazorchuk and Christoffer S{\"o}derberg}

\begin{abstract}
We investigate various ways to define an analogue of 
BGG category $\mathcal{O}$ for the non-semi-simple
Takiff extension of the Lie algebra $\mathfrak{sl}_2$.
We describe Gabriel quivers for blocks of these
analogues of category $\mathcal{O}$ and prove extension
fullness of one of them in the category of all modules.
\end{abstract}

\maketitle

\section{Introduction and description of the results}\label{s1}

The celebrated BGG category $\mathcal{O}$, introduced in \cite{BGG}, is 
originally associated to a triangular decomposition of a semi-simple 
finite dimensional complex Lie algebra. The definition of $\mathcal{O}$ is 
naturally generalized to all Lie algebras admitting some analogue of a 
triangular decomposition, see \cite{MP}. These include, in particular, 
Kac-Moody algebras and Virasoro algebra. Category $\mathcal{O}$ has a number
of spectacular properties and applications to various areas of mathematics,
see for example \cite{Hu} and references therein.

The paper \cite{DLMZ} took some first steps in trying to understand 
structure and properties of an analogue of category $\mathcal{O}$ 
in the case of a non-reductive finite dimensional Lie algebra. 
The investigation in \cite{DLMZ} focuses on category $\mathcal{O}$ for
the so-called Schr{\"o}dinger algebra, which is a central extension of
the semi-direct product of $\mathfrak{sl}_2$ with its natural $2$-dimensional
module. It turned out that, for the Schr{\"o}dinger algebra, the behavior 
of blocks of category $\mathcal{O}$   with non-zero central charge is exactly 
the same as the behavior of blocks of category $\mathcal{O}$ for the algebra 
$\mathfrak{sl}_2$. At the same, block with zero central charge turned out
to be significantly more difficult. For example, it was shown in \cite{DLMZ}
that some blocks of $\mathcal{O}$ for the Schr{\"o}dinger algebra have wild 
representation type, while all blocks of $\mathcal{O}$ for $\mathfrak{sl}_2$
have finite representation type.

In the present paper we look at a different non-reductive extension of 
the algebra $\mathfrak{sl}_2$, namely, the corresponding Takiff Lie algebra
$\mathfrak{g}$ defined as the semi-direct product of $\mathfrak{sl}_2$ 
with the adjoint  representation. Such Lie algebras were defined and studied by 
Takiff in \cite{Ta} with the primary interest coming from invariant theory. 
Alternatively, the Takiff Lie algebra $\mathfrak{g}$ can be described as
the tensor product $\mathfrak{sl}_2\otimes_{\mathbb{C}}\big(\mathbb{C}[x]/(x^2)\big)$.
The latter suggests an obvious generalization of the notion of a triangular
decomposition for $\mathfrak{g}$  by tensoring the components of a triangular
decomposition for $\mathfrak{sl}_2$ with $\mathbb{C}[x]/(x^2)$.

Having defined a triangular decomposition for $\mathfrak{g}$, we can define
Verma modules and try to guess a definition for category $\mathcal{O}$.
The latter turned out to be a subtle task as the most obvious definition
of category $\mathcal{O}$ does not really work as expected, in particular,
it does not contain Verma modules. This forced us to investigate two 
alternative definitions of category $\mathcal{O}$:
\begin{itemize}
\item the first one analogous to the definition of the so-called 
{\em thick category $\mathcal{O}$}, see, for example, \cite{Sor},
in which the action of the Cartan subalgebra is only expected to 
be locally finite and not necessarily semi-simple as in the classical definition;
\item and the second one given by the full subcategory of the thick category 
$\mathcal{O}$ from the first definition with the additional
requirement that the Cartan subalgebra of $\mathfrak{sl}_2$
acts diagonalizably.
\end{itemize}

The results of this paper fall into the following three categories:
\begin{itemize}
\item We describe the linkage between simple object in both our
versions of category $\mathcal{O}$ and in this
way explicitly determine all (indecomposable) blocks, see Theorem~\ref{thm9}.
\item We determine the Gabriel quivers for all blocks, see 
Corollaries~\ref{cor16}, \ref{cor23}, \ref{cor25}, \ref{cor33}.
\item We prove that thick category $\mathcal{O}$ is extension full
in the category of all $\mathfrak{g}$-modules, see Theorem~\ref{thm6}.
\end{itemize}
For some of the blocks, we also obtain not only a Gabriel quiver,
but also a fairly explicit description of the whole block,
see Theorem~\ref{thm15}. Some of the results are unexpected and
look rather surprizing. For example, the trivial $\mathfrak{g}$-module
exhibits behavior different from the behaviour of all other 
simple finite dimensional $\mathfrak{g}$-modules, compare
Lemma~\ref{lem27-1} and Proposition~\ref{prop27}.

The paper is organized as follows: All preliminaries are collected in
Section~\ref{s2}, in particular, in this section we define all
main protagonists of the paper and describe their basic properties.
In Section~\ref{s3} we prove extension fullness of thick category 
$\mathcal{O}$ in the category of all $\mathfrak{g}$-modules.
Section~\ref{s4} is devoted to the study of the decomposition of
both $\mathcal{O}$ and its thick version $\widetilde{\mathcal{O}}$ 
into indecomposable blocks.
As usual, generic Verma modules over $\mathfrak{g}$ are simple.
In Section~\ref{s6} we describe the structure of those Verma modules
that are not simple. Finally, in Section~\ref{s5}, we compute
first extensions between simple highest weight modules and in this
way determine the Gabriel quivers of all block in 
$\mathcal{O}$ and $\widetilde{\mathcal{O}}$. 

This paper is a revision, correction and extension of the master thesis 
\cite{Sod} of the second author written under the supervision of the 
first author.

\section{Takiff $\mathfrak{sl}_2$ and its modules}\label{s2}

\subsection{Takiff $\mathfrak{sl}_2$}\label{s2.1}

In this paper we work over the field $\mathbb{C}$ of complex numbers.
Consider the Lie algebra $\mathfrak{sl}_2$ with the standard basis
$\{e,h,f\}$ and the Lie bracket
\begin{displaymath}
[e,f]=h,\qquad [h,e]=2e,\qquad [h,f]=-2f. 
\end{displaymath}
Let $D:=\mathbb{C}[x]/(x^2)$ be the algebra of dual numbers.
Consider the associated {\em Takiff Lie algebra} 
$\mathfrak{g}=\mathfrak{sl}_2\otimes_{\mathbb{C}}D$
with the Lie bracket
\begin{displaymath}
[a\otimes x^i,b\otimes x^j]=[a,b]\otimes x^{i+j}, 
\end{displaymath}
where $a,b\in \mathfrak{sl}_2$ and $i,j\in\{0,1\}$ with
the Lie bracket on the right hand side being the $\mathfrak{sl}_2$-Lie bracket.
Set
\begin{displaymath}
\overline{e}:=e\otimes x,\qquad 
\overline{f}:=f\otimes x,\qquad 
\overline{h}:=h\otimes x. 
\end{displaymath}

Consider the standard triangular decomposition
\begin{displaymath}
\mathfrak{sl}_2=\mathfrak{n}_-\oplus\mathfrak{h}\oplus\mathfrak{n}_+, 
\end{displaymath}
where $\mathfrak{n}_-$ is generated by $f$, 
$\mathfrak{h}$ is generated by $h$ and $\mathfrak{n}_+$ is generated by $e$.
Let $\overline{\mathfrak{n}}_-$ be the subalgebra of $\mathfrak{g}$
generated by $e$ and $\overline{e}$, $\overline{\mathfrak{h}}$ 
the subalgebra of $\mathfrak{g}$ generated by $h$ and $\overline{h}$, and
$\overline{\mathfrak{n}}_+$ be the subalgebra of $\mathfrak{g}$
generated by $f$ and $\overline{f}$. The we have the following {\em triangular decomposition}
of $\mathfrak{g}$:
\begin{displaymath}
\mathfrak{g}=\overline{\mathfrak{n}}_-\oplus\overline{\mathfrak{h}}\oplus\overline{\mathfrak{n}}_+. 
\end{displaymath}
We set  $\mathfrak{b}:=\mathfrak{h}\oplus\mathfrak{n}_+$
and $\overline{\mathfrak{b}}:=\overline{\mathfrak{h}}\oplus
\overline{\mathfrak{n}}_+$.

For a Lie algebra $\mathfrak{a}$, we denote by $U(\mathfrak{a})$ the corresponding
universal enveloping algebra.

The natural projection $\mathfrak{g}\tto \mathfrak{sl}_2$ induced
an inclusion of $\mathfrak{sl}_2$-Mod to $\mathfrak{g}$-Mod,
which we denote by $\iota$.

By a direct calculation, it is easy to check that the {\em Casimir element}
\begin{equation}\label{eq4}
\mathtt{c}:=h\overline{h}+2\overline{h}+2f\overline{e}+2\overline{f}e 
\end{equation}
belongs to the center of $U(\mathfrak{g})$, see Example~1.2 in \cite{Mo}.

\subsection{(Generalized) weight modules}\label{s2.2}

A $\mathfrak{g}$-module $M$ is called a {\em generalized weight module}
provided that the action of $U(\overline{\mathfrak{h}})$ on $M$
is locally finite. As $U(\overline{\mathfrak{h}})$ is just the polynomial
algebra in $h$ and $\overline{h}$, every generalized weight module $M$
admits a decomposition
\begin{displaymath}
M=\bigoplus_{\lambda\in\overline{\mathfrak{h}}^*}M^{\lambda}, 
\end{displaymath}
where $M^{\lambda}$ denotes the set of all vectors in $M$
killed by some power of the maximal ideal $\mathbf{m}_{\lambda}$
of $U(\overline{\mathfrak{h}})$ corresponding to $\lambda$.
We will say that a generalized weight module $M$ is a 
{\em weight module} provided that the action of $h$ on $M$
is diagonalizable. We will say that a weight module $M$ is a 
{\em strong weight module} provided that the action of 
$\overline{\mathfrak{h}}$ on $M$ is diagonalizable. 

Note that submodules, quotients and extensions of generalized weight 
modules are generalized weight modules. Also submodules
and quotients of (strong) weight modules are (strong) weight modules.
From the commutation relations in $\mathfrak{g}$, 
for any generalized weight modules $M$ and any 
$\lambda\in\overline{\mathfrak{h}}^*$, we have
\begin{equation}\label{eq1}
\overline{\mathfrak{h}}M^{\lambda}\subset M^{\lambda},\qquad
\overline{\mathfrak{n}}_+M^{\lambda}\subset M^{\lambda+\alpha},\qquad
\overline{\mathfrak{n}}_-M^{\lambda}\subset M^{\lambda-\alpha},
\end{equation}
where $\alpha\in \overline{\mathfrak{h}}^*$ is given by
$\alpha(h)=2$, $\alpha(\overline{h})=0$.

If $M$ is a generalized weight module, then the set of all $\lambda$
for which $M^{\lambda}\neq 0$ is called the {\em support} of $M$
and denoted $\mathrm{supp}(M)$.

\subsection{Verma modules}\label{s2.3}

For a fixed $\lambda\in \overline{\mathfrak{h}}^*$, we have the
corresponding simple $U(\overline{\mathfrak{h}})$-module
$\mathbb{C}_{\lambda}:=U(\overline{\mathfrak{h}})/\mathbf{m}_{\lambda}$.
Setting $\overline{\mathfrak{n}}_+\mathbb{C}_{\lambda}=0$ defines
on $\mathbb{C}_{\lambda}$ the structure of a $\overline{\mathfrak{b}}$-module.
The $\mathfrak{g}$-module
\begin{displaymath}
\Delta(\lambda):=
\mathrm{Ind}^{U(\mathfrak{g})}_{U(\overline{\mathfrak{b}})}\, 
\big(\mathbb{C}_{\lambda}\big)\cong
U(\mathfrak{g})\bigotimes_{U(\overline{\mathfrak{b}})}\mathbb{C}_{\lambda}
\end{displaymath}
is called the {\em Verma module} associated to $\lambda$.
The standard argument, see Proposition~7.1.11 in \cite{Di}, shows that 
$\Delta(\lambda)$ has a unique simple quotient. We denote this
simple quotient of $\Delta(\lambda)$ by $L(\lambda)$.
From the PBW Theorem and formula \eqref{eq1}, it follows that
\begin{equation}\label{eq5}
\mathrm{supp}(\Delta(\lambda))=\{\lambda-n\alpha\,:\,n\in\mathbb{Z}_{\geq 0}\}. 
\end{equation}
In fact, from the PBW Theorem, it follows that, for 
$n\in \mathbb{Z}_{\geq 0}$, we have
\begin{equation}\label{eq2}
\dim\left(\Delta(\lambda)^{\lambda-n\alpha}\right)=n+1 
\end{equation}
as the elements $\{f^i\overline{f}^{n-i}v_{\lambda}\,:\, i=0,1,\dots,n\}$, 
where $v_{\lambda}$ denotes  the canonical generator of $\Delta(\lambda)$, 
form a basis of $\Delta(\lambda)^{\lambda-n\alpha}$.
The weight $\lambda$ is the {\em highest weight} of
$\Delta(\lambda)$.

The following simplicity criterion for $\Delta(\lambda)$
can be deduced from the main result of \cite{Wi}, however, we
include a short proof for the sake of completeness.

\begin{proposition}\label{prop1}
The module  $\Delta(\lambda)$ is simple if and only if 
$\lambda(\overline{h})\neq 0$.
\end{proposition}

\begin{proof}
Let $v_{\lambda}$ be the canonical generator of $\Delta(\lambda)$.
Assume first that $\lambda(\overline{h})=0$ and consider the element
$w=\overline{f}v_{\lambda}$. Then we have $ew=\overline{e}w=0$
and hence, from the PBW Theorem and \eqref{eq1}, it follows that 
the submodule $N$ in $\Delta(\lambda)$ generated by $w$
satisfies $N^{\lambda}=0$ and thus is a non-zero proper submodule.
Therefore $\Delta(\lambda)$ is reducible in this case.

Now assume that $\lambda(\overline{h})\neq 0$. We need to show that 
any non-zero submodule $N$ of $\Delta(\lambda)$ contains $v_{\lambda}$.
If $N^{\lambda}\neq 0$, then the fact that $v_{\lambda}\in N$ is clear.
Assume now that $N^{\lambda}=0$ and let $n\in\mathbb{Z}_{>0}$ be 
minimal such that $N^{\lambda-n\alpha}\neq 0$. Let $w\in N^{\lambda-n\alpha}$
be a non-zero element. Using the PBW Theorem, we may write
\begin{displaymath}
w=\sum_{i=0}^nc_if^i\overline{f}^{n-i}v_{\lambda}.
\end{displaymath}
Denote the maximal value of $i$ such that $c_{i}\neq 0$ by $k$. 
Then it is easy to check that $(\overline{h}-\lambda(\overline{h}))^kw$ equals
$\overline{f}^{n}v_{\lambda}$ up to a non-zero scalar.
In particular, $N$ contains $\overline{f}^{n}v_{\lambda}$.
But then it is easy to check that
$e\overline{f}^{n}v_{\lambda}$ equals $\overline{f}^{n-1}v_{\lambda}$,
up to a non-zero scalar. In particular, $N^{\lambda-n\alpha+\alpha}\neq 0$.
The obtained contradiction proves that 
$N^{\lambda}\neq 0$ and the claim of the proposition follows.
\end{proof}

For $\mu\in\mathfrak{h}^*$, we denote by $\Delta^{\mathfrak{sl}_2}(\mu)$
the $\mathfrak{sl}_2$-Verma module with highest weight $\mu$ and
by $L^{\mathfrak{sl}_2}(\mu)$ the unique simple quotient of 
$\Delta^{\mathfrak{sl}_2}(\mu)$. From Proposition~\ref{prop1},
we  obtain the following corollary.

\begin{corollary}\label{cor2}
For $\lambda\in\overline{\mathfrak{h}}^*$, we have
\begin{displaymath}
L(\lambda)\cong
\begin{cases}
\Delta(\lambda),& \text{ if }\lambda(\overline{h})\neq 0;\\
\iota(L^{\mathfrak{sl}_2}(\lambda\vert_{\mathfrak{h}})),
&\text{ if }\lambda(\overline{h})= 0.
\end{cases}
\end{displaymath}
\end{corollary}

\begin{proof}
If $\lambda(\overline{h})\neq 0$, then the claim is just a part of
Proposition~\ref{prop1}. If $\lambda(\overline{h})=0$, then
the unique up to scalar non-zero vector in 
$\iota(L^{\mathfrak{sl}_2}(\lambda\vert_{\mathfrak{h}}))^{\lambda}$
generates a $\overline{\mathfrak{b}}$-submodule of 
$\iota(L^{\mathfrak{sl}_2}(\lambda\vert_{\mathfrak{h}}))$
isomorphic to $\mathbb{C}_{\lambda}$. By adjunction,
we obtain a non-zero homomorphism from $\Delta(\lambda)$
to $\iota(L^{\mathfrak{sl}_2}(\lambda\vert_{\mathfrak{h}}))$
which must be surjective as the latter module is simple.
Consequently, 
$\iota(L^{\mathfrak{sl}_2}(\lambda\vert_{\mathfrak{h}}))$ 
must be isomorphic to $L(\lambda)$ by the definition of $L(\lambda)$.
\end{proof}

\subsection{(Thick) category $\mathcal{O}$}\label{s2.4}

We define  {\em thick category $\mathcal{O}$}, denoted
$\widetilde{\mathcal{O}}$, as the full subcategory of the category
of all finitely generated $\mathfrak{g}$-modules consisting of
all $\mathfrak{g}$-modules the action of 
$U(\overline{\mathfrak{b}})$ on which is locally finite.
Note that, by definition, all modules in $\mathcal{O}$ are
generalized weight modules.

We define  {\em classical category $\mathcal{O}$}, denoted
$\mathcal{O}$, as the full subcategory of $\widetilde{\mathcal{O}}$
consisting of all weight modules.
Finally, we define  {\em strong category $\mathcal{O}$}, denoted
$\underline{\mathcal{O}}$, as the full subcategory of $\mathcal{O}$
consisting of all strong weight modules.

As $U(\mathfrak{g})$ is noetherian, the categories 
$\widetilde{\mathcal{O}}$, $\mathcal{O}$ and  $\underline{\mathcal{O}}$
are abelian categories closed under taking submodules, quotients 
and finite direct sums. Directly from the definition, it also 
follows that $\widetilde{\mathcal{O}}$ is closed under taking extensions, 
in particular, $\widetilde{\mathcal{O}}$ is  a Serre subcategory of 
$\mathfrak{g}$-mod.

\begin{proposition}\label{prop4}
For each $\lambda\in\overline{\mathfrak{h}}^*$, the module
$\Delta(\lambda)$ belongs to both $\widetilde{\mathcal{O}}$ and
$\mathcal{O}$. However, $\Delta(\lambda)$ does not belong to 
$\underline{\mathcal{O}}$.
\end{proposition}

\begin{proof}
That $\Delta(\lambda)\in \widetilde{\mathcal{O}}$ follows from \eqref{eq2}.
That $\Delta(\lambda)\in {\mathcal{O}}$ follows by combining the
fact that $\Delta(\lambda)\in \widetilde{\mathcal{O}}$ and that
the adjoint action of $h$ on $U(\mathfrak{g})$ is diagonalizable
(implying that the action of $h$ on $\Delta(\lambda)$ is diagonalizable).

That $\Delta(\lambda)\not\in \underline{\mathcal{O}}$  follows
from the fact that the matrix of the action of $\overline{h}$ in
the basis $\{fv_{\lambda},\overline{f}v_{\lambda}\}$ of
$\Delta(\lambda)^{\lambda-\alpha}$, where $v_{\lambda}$ is the canonical
generator of $\Delta(\lambda)$, has the form
\begin{displaymath}
\left( 
\begin{array}{cc}
\lambda(\overline{h})&0\\
-2&\lambda(\overline{h})
\end{array}
\right) 
\end{displaymath}
and hence is not diagonalizable.
\end{proof}

\begin{proposition}\label{prop3}
{\hspace{2mm}}

\begin{enumerate}[$($a$)$]
\item\label{prop3.1} The set $\{L(\lambda)\,:\,
\lambda\in\overline{\mathfrak{h}}^*\}$ is a complete and irredundant
list of representatives of isomorphism classes of simple objects
in $\widetilde{\mathcal{O}}$.
\item\label{prop3.2} The set $\{L(\lambda)\,:\,
\lambda\in\overline{\mathfrak{h}}^*\}$ is a complete and irredundant
list of representatives of isomorphism classes of simple objects
in ${\mathcal{O}}$.
\item\label{prop3.3} The set $\{\iota(L^{\mathfrak{sl}_2}(\mu))\,:\,
\mu\in{\mathfrak{h}}^*\}$ is a complete and irredundant
list of representatives of isomorphism classes of simple objects
in $\underline{\mathcal{O}}$.
\end{enumerate}
\end{proposition}

\begin{proof}
Let $L$ be a simple module in $\widetilde{\mathcal{O}}$ and
$v$ a non-zero element in $L$. Since the vector space
$U(\overline{\mathfrak{b}})v$ is finite dimensional, it contains
a non-zero element $w$ such that $\overline{\mathfrak{n}}_+w=0$,
$hw=\lambda(h)w$ and $\overline{h}w=\lambda(\overline{h})w$,
for some $\lambda\in \overline{\mathfrak{h}}^*$. Then
$\mathbb{C}w$, is isomorphic, as a $\overline{\mathfrak{b}}$-module,
to $\mathbb{C}_{\lambda}$. By Proposition~\ref{prop4},
we have $\Delta(\lambda)\in \widetilde{\mathcal{O}}$.
By adjunction, we obtain a non-zero
homomorphism from $\Delta(\lambda)$ to $L$. This implies 
$L\cong L(\lambda)$ and proves claim~\eqref{prop3.1}.
Claim~\eqref{prop3.2} is proved similarly.

To prove claim~\eqref{prop3.3}, we can use 
claim~\eqref{prop3.2} and hence just need to check, for which
$\lambda\in \overline{\mathfrak{h}}^*$, the module
$L(\lambda)$ belongs to $\underline{\mathcal{O}}$.
If $\lambda(\overline{h})\neq 0$, then $L(\lambda)=\Delta(\lambda)$
by Corollary~\ref{cor2} and hence $L(\lambda)\not\in \underline{\mathcal{O}}$
by Proposition~\ref{prop4}. If $\lambda(\overline{h})=0$, 
then $L(\lambda)=\iota(L^{\mathfrak{sl}_2}(\lambda\vert_{\mathfrak{h}})$
by Corollary~\ref{cor2} and 
$\iota(L^{\mathfrak{sl}_2}(\lambda\vert_{\mathfrak{h}})\in 
\underline{\mathcal{O}}$ since the action of $\overline{h}$
on $\iota(L^{\mathfrak{sl}_2}(\lambda\vert_{\mathfrak{h}})$
is zero and thus diagonalizable. This completes the proof.
\end{proof}

Proposition~\ref{prop3} has the following consequence.

\begin{corollary}\label{cor5}
The functor $\iota$ induces an equivalence between the
category $\mathcal{O}$ for $\mathfrak{sl}_2$ and the 
category $\underline{\mathcal{O}}$.
\end{corollary}

\begin{proof}
By construction, the functor $\iota$ is full and faithful and maps
the category $\mathcal{O}$ for $\mathfrak{sl}_2$ to the 
category $\underline{\mathcal{O}}$. Hence, what we need to prove
is that this  restriction of $\iota$ is dense. By 
Proposition~\ref{prop3}\eqref{prop3.3}, $\iota$ hits all 
isomorphism classes of simple objects in $\underline{\mathcal{O}}$.
In particular, $\overline{h}$ annihilates all simple objects in
$\underline{\mathcal{O}}$. Since, by the definition of 
$\underline{\mathcal{O}}$, the action of $\overline{h}$ on any object 
in $\underline{\mathcal{O}}$ is semi-simple, it follows that 
$\overline{h}$ annihilates all objects in $\underline{\mathcal{O}}$.

Since the ideal of $\mathfrak{g}$ generated by $\overline{h}$ contains
both $\overline{e}$ and $\overline{f}$, it follows that the latter two
elements annihilate all object in $\underline{\mathcal{O}}$. This
yields that every object in $\underline{\mathcal{O}}$ is, in fact, 
isomorphic to an object in the image of $\iota$. The claim follows.
\end{proof}

Due to Corollary~\ref{cor5}, the category $\underline{\mathcal{O}}$
is fairly well-understood, see e.g. \cite{Ma} for a very detailed description.
Therefore, in what follows, we focus on studying the categories
$\widetilde{\mathcal{O}}$ and $\mathcal{O}$.

\section{Extension fullness of $\widetilde{\mathcal{O}}$ in $\mathfrak{g}$-Mod}\label{s3}

The inclusion functor $\Phi:\widetilde{\mathcal{O}}\hookrightarrow \mathfrak{g}$-Mod
is exact and hence induces, for each $M,N\in \widetilde{\mathcal{O}}$ and $i\geq 0$,
homomorphisms
\begin{displaymath}
\varphi_{M,N}^{(i)}:\mathrm{Ext}^i_{\widetilde{\mathcal{O}}}(M,N)\to
\mathrm{Ext}^i_{\mathfrak{g}\text{-}\mathrm{Mod}}(M,N)
\end{displaymath}
of abelian groups. As $\widetilde{\mathcal{O}}$ is a full subcategory of $\mathfrak{g}$-Mod,
all $\varphi_{M,N}^{(0)}$ are isomorphisms. As $\widetilde{\mathcal{O}}$ is a Serre 
subcategory of $\mathfrak{g}$-Mod, all $\varphi_{M,N}^{(1)}$ are isomorphisms. The main
result of this section is the following statement.

\begin{theorem}\label{thm6}
The category  $\widetilde{\mathcal{O}}$ is {\em extension full} in $\mathfrak{g}$-Mod
in the sense that all $\varphi_{M,N}^{(i)}$ are isomorphisms.
\end{theorem}

Theorem~\ref{thm6} is a generalization of Theorem~16 in \cite{CM2} to our setup.
We refer the reader to \cite{CM1} and \cite{CM2} for more details on extension full subcategories.

\begin{proof}
We follow the proof of Theorem~2 in \cite{CM1}. Denote by $\hat{\tilde{\mathcal{O}}}$
the full subcategory of $\mathfrak{g}$-Mod consisting of all modules, the action of
$U(\overline{\mathfrak{b}})$ on which is locally finite. The difference between 
$\hat{\tilde{\mathcal{O}}}$ and $\tilde{\mathcal{O}}$ is that, in the case of 
$\hat{\tilde{\mathcal{O}}}$, we drop the condition on modules to be finitely generated.

First we note that $\tilde{\mathcal{O}}$ is extension full in $\hat{\tilde{\mathcal{O}}}$.
Indeed, if $M\in \hat{\tilde{\mathcal{O}}}$, $N\in \tilde{\mathcal{O}}$ and
$\alpha:M\to N$ is a surjective homomorphism, we can use that $N$ is finitely generated
to claim that $N$ is in the image of a finitely generated submodule of $M$. 
Therefore the fact that $\tilde{\mathcal{O}}$ is extension full in 
$\hat{\tilde{\mathcal{O}}}$ follows from Proposition~3 in \cite{CM1}
(applied in the situation $\mathcal{B}=\tilde{\mathcal{O}}$ and
$\mathcal{A}=\hat{\tilde{\mathcal{O}}}$).

To complete the proof of the theorem, it remains to prove that 
$\hat{\tilde{\mathcal{O}}}$ is extension full in $\mathfrak{g}$-Mod.
For a locally finite dimensional $U(\overline{\mathfrak{b}})$-module $V$,
denote by $M(V)$ the induced module 
$\mathrm{Ind}_{U(\overline{\mathfrak{b}})}^{U(\mathfrak{g})}(V)$.
Note that, by Theorem~6 in \cite{BM}, the action of $U(\overline{\mathfrak{b}})$
on $M(V)$ is locally finite. The same computation as in the proof of
Lemma~3 in \cite{CM1} shows that, for any $V$ as above, any 
$N\in \hat{\tilde{\mathcal{O}}}$ and any $i\geq 0$, the natural map
\begin{displaymath}
\mathrm{Ext}^i_{\hat{\widetilde{\mathcal{O}}}}(M(V),N)\to
\mathrm{Ext}^i_{\mathfrak{g}\text{-}\mathrm{Mod}}(M(V),N)
\end{displaymath}
is an isomorphism. Therefore the extension fullness of 
$\hat{\tilde{\mathcal{O}}}$ in $\mathfrak{g}$-Mod
follows from Proposition~1 in \cite{CM1} (applied in the situation 
$\mathcal{A}=\mathfrak{g}$-Mod, $\mathcal{B}=\hat{\tilde{\mathcal{O}}}$ and
$\mathcal{B}_0$ consisting of modules of the form $M(V)$, for $V$ as above). 
This completes the proof.
\end{proof}

\section{Description of blocks}\label{s4}

\subsection{Characters and composition multiplicities}\label{s4.0}

\begin{lemma}\label{lem11}
Let  $\mathcal{X}\in\{\mathcal{O},\widetilde{\mathcal{O}}\}$
and $M\in \mathcal{X}$. There exist $k\in \mathbb{Z}_{>0}$
and $\lambda_1,\lambda_2,\dots,\lambda_k\in \overline{\mathfrak{h}}^*$
such that
\begin{equation}\label{eq7}
\mathrm{supp}(M)\subset\bigcup_{i=1}^k\{\lambda_i-\mathbb{Z}_{\geq 0}\alpha\},  
\end{equation}
moreover, for each $\mu\in \mathrm{supp}(M)$, the space
$M^{\mu}$ is finite dimensional.
\end{lemma}

\begin{proof}
If two modules $M_1$ and $M_2$ have the properties described in the
formulation of the lemma, then any extension of $M_1$ and $M_2$
also has similar properties. By definition, $M$ is finitely generated,
and hence, taking the first sentence into account, without loss
of generality we may assume that $M$ is generated by one element
$v\in M^{\nu}$, for some $\nu\in \overline{\mathfrak{h}}^*$.

The vector space $U(\overline{\mathfrak{b}})v$ is finite dimensional
and $\overline{\mathfrak{h}}$-stable. Hence the
$\overline{\mathfrak{h}}$-module $U(\overline{\mathfrak{b}})v$
has finite support, say $\{\lambda_1,\lambda_2,\dots,\lambda_k\}$.
By the PBW Theorem, we have the decomposition  $U(\mathfrak{g})=
U(\overline{\mathfrak{n}}_-)U(\overline{\mathfrak{b}})$.
Hence $M=U(\overline{\mathfrak{n}}_-)\big(U(\overline{\mathfrak{b}})v\big)$,
implying Formula~\eqref{eq7}. Moreover, since, considered as
an adjoint $\overline{\mathfrak{h}}$-module, all generalized weight
spaces of $U(\overline{\mathfrak{n}}_-)$ are finite dimensional,
it follows that all $M^{\mu}$ are finite dimensional. This completes the proof.
\end{proof}

For a finite subset $\boldsymbol{\mu}\subset \overline{\mathfrak{h}}^*$,
set
\begin{displaymath}
\overline{\boldsymbol{\mu}}=\bigcup_{\mu\in \boldsymbol{\mu}}
\{\mu-\mathbb{Z}_{\geq 0}\alpha\}.
\end{displaymath}
We write $\boldsymbol{\mu}\preceq \boldsymbol{\nu}$ provided that
$\overline{\boldsymbol{\mu}}\subset \overline{\boldsymbol{\nu}}$.

Consider the set $\mathbf{F}$ of all functions 
$\chi:\overline{\mathfrak{h}}^*\to\mathbb{Z}_{\geq 0}$ having the property 
that the {\em support} 
$\{\lambda\in \overline{\mathfrak{h}}^*\,:\,\chi(\lambda)\neq 0\}$
of $\chi$ belongs to $\overline{\boldsymbol{\mu}}$, for some $\boldsymbol{\mu}$
as above. The set $\mathbf{F}$ has the natural structure of an additive monoid with
respect to the pointwise addition of functions. The neutral element of this
monoid is the zero function.

Let $\mathcal{X}\in\{\mathcal{O},\widetilde{\mathcal{O}}\}$. Given
$M\in \mathcal{X}$, we define the {\em character} $\mathbf{ch}(M)$ as
the function from $\overline{\mathfrak{h}}^*$ to $\mathbb{Z}_{\geq 0}$
sending $\lambda$ to $\dim(M^{\lambda})$. By Lemma~\ref{lem11},
we have $\mathbf{ch}(M)\in \mathbf{F}$. Clearly, characters are
additive on short exact sequences, that is, for any
short exact sequence $0\to K\to M\to N\to O$ in $\mathcal{X}$, we have
$\mathbf{ch}(M)=\mathbf{ch}(K)+\mathbf{ch}(N)$.

\begin{proposition}\label{prop12}
Let $\mathcal{X}\in\{\mathcal{O},\widetilde{\mathcal{O}}\}$. 
\begin{enumerate}[$($a$)$]
\item\label{prop12.2} For any $M\in\mathcal{X}$, 
there are uniquely determined $\mathbf{k}_{\lambda}(M)\in \mathbb{Z}_{\geq 0}$,
where $\lambda\in \overline{\mathfrak{h}}^*$,
such that
\begin{displaymath}
\mathbf{ch}(M)=\sum_{\lambda\in\overline{\mathfrak{h}}^*}\mathbf{k}_{\lambda}(M)
\mathbf{ch}(L(\lambda)).
\end{displaymath}
\item\label{prop12.1} For every $\lambda\in \overline{\mathfrak{h}}^*$,
the function 
$\mathbf{k}_{\lambda}:\mathrm{Ob}(\mathcal{X})\to\mathbb{Z}_{\geq 0}$
has the following properties:
\begin{enumerate}[$($i$)$]
\item\label{prop12.1.1} $\mathbf{k}_{\lambda}(L(\lambda))=1$;
\item\label{prop12.1.2} $\mathbf{k}_{\lambda}(L(\mu))=0$, if $\mu\neq\lambda$;
\item\label{prop12.1.3} $\mathbf{k}_{\lambda}(M)=0$, if $\lambda\not\in\mathrm{supp}(M)$;
\item\label{prop12.1.4} $\mathbf{k}_{\lambda}(M)$ is additive on short exact sequences.
\end{enumerate}
\end{enumerate}
\end{proposition}

\begin{proof}
Clearly, $\mathbf{k}_{\lambda}(M)=0$ if $\lambda\not\in\mathrm{supp}(M)$,
and hence the sum in \eqref{prop12.2} can be taken over $\mathrm{supp}(M)$
instead of the whole $ \overline{\mathfrak{h}}^*$.

Assume first that, for $M\in\mathcal{X}$, we have
\begin{displaymath}
\mathbf{ch}(M)=\sum_{\lambda\in\mathrm{supp}(M)}a_{\lambda}
\mathbf{ch}(L(\lambda))=
\sum_{\lambda\in\mathrm{supp}(M)}b_{\lambda}
\mathbf{ch}(L(\lambda)),
\end{displaymath}
where all $a_{\lambda}$ and $b_{\lambda}$ are in $\mathbb{Z}_{\geq 0}$.
Assume that there is some $\lambda$ such that $a_{\lambda}\neq b_{\lambda}$.
Let $X:=\{\lambda\,:\, a_{\lambda}>b_{\lambda}\}$ and
$Y:=\mathrm{supp}(M)\setminus X$. Then we have
\begin{displaymath}
\chi:=\sum_{\lambda\in X}(a_{\lambda}-b_{\lambda})
\mathbf{ch}(L(\lambda))=
\sum_{\mu\in Y}(b_{\mu}-a_{\mu})
\mathbf{ch}(L(\mu)).
\end{displaymath}
By our assumptions, $\chi\in\mathbf{F}$ is non-zero. Then there exists
$\nu\in \overline{\mathfrak{h}}^*$ such that $\chi(\nu)\neq 0$ but
$\chi(\nu+m\alpha)=0$, for all $m\in\mathbb{Z}_{>0}$. If $\nu\in X$,
then  $\nu\not\in Y$ and from the property 
$\chi(\nu+m\alpha)=0$, for all $m\in\mathbb{Z}_{>0}$, we see that 
$\chi(\nu)\neq 0$ is not possible if we compute $\chi$ using the
second expression. Similarly, if $\nu\in Y$,
then  $\nu\not\in X$ and we get that $\chi(\nu)\neq 0$ is not 
possible if we compute $\chi$ using the first expression. 
The obtained contradiction shows that, if a decomposition of the form
as in \eqref{prop12.2} exists, then it is unique. 

Let us now prove existence of \eqref{prop12.2}.
If $M=0$, we set $\mathbf{k}_{\lambda}(M)=0$, for all $\lambda$.
Let $M\in\mathcal{X}$ be non-zero and $\lambda\in\mathrm{supp}(M)$
be such that $\lambda+m\alpha\not\in\mathrm{supp}(M)$, for all $m\in\mathbb{Z}_{>0}$.
Let $v\in M^{\lambda}$ be a non-zero element which is an eigenvector
for both $h$ and $\overline{h}$ and set $K:=U(\mathfrak{g})v\subset M$
and $N:=M/K$. By adjunction, there is a non-zero epimorphism from
$\Delta(\lambda)$ to $K$ sending the canonical generator of 
$\Delta(\lambda)$ to $v$. Let $K'$ denote the image, under this epimorphism,
of the unique maximal submodule of $\Delta(\lambda)$. By construction, we have two
short exact sequences:
\begin{displaymath}
0\to K'\to K\to L(\lambda)\to 0\qquad\text{ and }\qquad
0\to K\to M\to N\to 0.
\end{displaymath}

For each $\mu\in\mathrm{supp}(M)$, define
\begin{equation}\label{eq8}
\mathbf{k}_{\mu}(M)=
\begin{cases}
\mathbf{k}_{\mu}(K')+\mathbf{k}_{\mu}(N),& \text{ if }\mu\neq\lambda;\\
\mathbf{k}_{\mu}(K')+\mathbf{k}_{\mu}(N)+1,& \text{ if }\mu=\lambda.
\end{cases}
\end{equation}
Note that $\mathrm{supp}(N)\subset \mathrm{supp}(M)$ and
$\dim(N^{\lambda})<\dim(M^{\lambda})$, moreover, we also have 
$\mathrm{supp}(K')\subset \mathrm{supp}(M)$ and
$\lambda\not\in \mathrm{supp}(K')$. Therefore, thanks to Lemma~\ref{lem11},
Formula~\eqref{eq8} gives an iterative procedure which, after a finite number of 
iterations, completely determines $\mathbf{k}_{\mu}(M)$ such that 
\eqref{prop12.2} holds by construction.

It remains to check that $\mathbf{k}_{\mu}(M)$ defined above have all the 
properties listed in \eqref{prop12.1}. 
Properties~\eqref{prop12.1.1}-\eqref{prop12.1.3} follow directly from 
the definition in the previous paragraph. Property~\eqref{prop12.1.1} follows from the 
equality in \eqref{prop12.2} and the fact that characters are additive
on short exact sequences.
\end{proof}

The number $\mathbf{k}_{\mu}(M)$ will be called the {\em composition
multiplicity} of $L(\mu)$ in $M$.

\subsection{Some first extensions between simple objects}\label{s4.1}

\begin{proposition}\label{prop7}
Let $\lambda,\mu\in\overline{\mathfrak{h}}^*$ be such that $\lambda \neq \mu$
and $\lambda(\overline{h})\neq 0$. Then, for any $\mathcal{X}\in\{\mathcal{O},\widetilde{\mathcal{O}}\}$, we have
\begin{displaymath}
\mathrm{Ext}^1_{\mathcal{X}} (L(\lambda),L(\mu))=
\mathrm{Ext}^1_{\mathcal{X}} (L(\mu),L(\lambda))=0 .
\end{displaymath}
\end{proposition}

\begin{proof}
We prove that $\mathrm{Ext}^1_{\mathcal{X}} (L(\lambda),L(\mu))=0$, the second 
claim is similar. Assume that 
\begin{equation}\label{eq3}
0\to L(\mu)\to M\to L(\lambda)\to 0 
\end{equation}
is a short exact sequence in $\mathcal{X}$. Note that 
\begin{displaymath}
\mathrm{supp}(M)= 
\mathrm{supp}(L(\lambda))\bigcup
\mathrm{supp}(L(\mu))\subset
\{\lambda-\mathbb{Z}_{\geq 0}\alpha\}\bigcup
\{\mu-\mathbb{Z}_{\geq 0}\alpha\}
\end{displaymath}
by \eqref{eq5}. If $\mu\not\in\lambda+\mathbb{Z}\alpha$, then
$M^{\lambda+\alpha}=0$ and hence $\overline{\mathbf{n}}_+M^{\lambda}=0$.
By adjunction, this gives as a non-zero homomorphism from $\Delta(\lambda)=L(\lambda)$
to $M$ which splits \eqref{eq3}. This implies the necessary claim in case
$\mu\not\in\lambda+\mathbb{Z}\alpha$.

If $\mu\in\lambda+\mathbb{Z}\alpha$, then $\mu(\overline{h})=\lambda(\overline{h})\neq 0$.
By applying the Casimir element $\mathtt{c}$, see \eqref{eq4}, to the highest
weight elements in $L(\lambda)$ and $L(\mu)$, we see that $\mathtt{c}$ acts as the
scalar $\lambda(\overline{h})(\lambda(h)+2)$ on $L(\lambda)$ and
as the scalar $\mu(\overline{h})(\mu(h)+2)$ on $L(\mu)$.
As $\mu(\overline{h})=\lambda(\overline{h})\neq 0$ but 
$\mu\neq\lambda$, we obtain
\begin{displaymath}
\lambda(\overline{h})(\lambda(h)+2)\neq \mu(\overline{h})(\mu(h)+2).  
\end{displaymath}
This means that $L(\lambda)$ and $L(\mu)$ have different central characters
and hence \eqref{eq3} splits. The claim of the proposition follows.
\end{proof}

Following the proof of Proposition~\ref{prop7}, we also obtain the following claim.

\begin{corollary}\label{cor8}
Let $\lambda,\mu\in\overline{\mathfrak{h}}^*$ be such that 
$\lambda\not\in\mu+\mathbb{Z}\alpha$. Then, for any 
$\mathcal{X}\in\{\mathcal{O},\widetilde{\mathcal{O}}\}$, we have 
\begin{displaymath}
\mathrm{Ext}^1_{\mathcal{X}} (L(\lambda),L(\mu))=
\mathrm{Ext}^1_{\mathcal{X}} (L(\mu),L(\lambda))=0 .
\end{displaymath} 
\end{corollary}

\subsection{Easy blocks}\label{s4.2}
Let $\mathcal{X}\in\{\mathcal{O},\widetilde{\mathcal{O}}\}$.
Set 
\begin{displaymath}
\overline{\mathfrak{h}}^*_1:=
\{\lambda\in \overline{\mathfrak{h}}^*\,:\, \lambda(\overline{h})\neq 0\},\qquad
\overline{\mathfrak{h}}^*_0:=\overline{\mathfrak{h}}^*\setminus
\overline{\mathfrak{h}}^*_1.
\end{displaymath}
For $\lambda\in \overline{\mathfrak{h}}^*_1$, denote by 
$\mathcal{X}(\lambda)$ the full subcategory of $\mathcal{X}$ consisting of
all modules with support $\{\lambda-\mathbb{Z}_{\geq 0}\alpha\}$.

\subsection{Difficult blocks}\label{s4.3}
For $\xi\in \overline{\mathfrak{h}}^*_0/\mathbb{Z}\alpha$, 
denote by $\mathcal{X}(\xi)$ the full subcategory of $\mathcal{X}$ consisting of
all modules whose support is contained in $\xi$.

\subsection{Block decomposition}\label{s4.4}

\begin{theorem}\label{thm9}
For $\mathcal{X}\in\{\mathcal{O},\widetilde{\mathcal{O}}\}$, we have a decomposition
\begin{equation}\label{eq6}
\mathcal{X}=
\bigoplus_{\lambda\in \overline{\mathfrak{h}}^*_1}\mathcal{X}(\lambda)
\oplus
\bigoplus_{\xi\in \overline{\mathfrak{h}}^*_0/\mathbb{Z}\alpha}
\mathcal{X}(\xi)
\end{equation}
of $\mathcal{X}$ into a direct sum of indecomposable abelian 
subcategories (blocks).
\end{theorem}

\begin{proof}
Let $M\in \mathcal{X}$ be an indecomposable module. Then there is
$\lambda\in \overline{\mathfrak{h}}^*$ such that 
$\mathrm{supp}(M)\subset \xi:=\lambda+\mathbb{Z}\alpha$.
If $\lambda(\overline{h})=0$, then, by definition,
$M\subset \mathcal{X}(\xi)$. If $\lambda(\overline{h})\neq 0$,
then, by Proposition~\ref{prop7}, all composition subquotients of
$M$ are isomorphic to some $L(\mu)$, where $\mu\in \xi$.
Therefore $M\in \mathcal{X}(\mu)$. This implies existence of the
direct sum decomposition as in \eqref{eq6}.

It remains to prove that all summands in the right hand side of 
\eqref{eq6} are indecomposable. That each $\mathcal{X}(\lambda)$,
where $\lambda\in \overline{\mathfrak{h}}^*_1$, is indecomposable,
is clear as $\mathcal{X}(\lambda)$ contains, by construction,
only one simple module, up to isomorphism.

Let us argue that each $\mathcal{X}(\xi)$, where 
$\xi\in \overline{\mathfrak{h}}^*_0/\mathbb{Z}\alpha$, is indecomposable.
For this it is enough to show that, for every $\lambda\in\xi$,
there is an indecomposable module $M\in \mathcal{X}(\xi)$ 
such that both $\mathbf{k}_{\lambda}(M)$ and $\mathbf{k}_{\lambda-\alpha}(M)$
are non-zero. Take $M=\Delta(\lambda)$. The module $\Delta(\lambda)$ is 
indecomposable as it has simple top. Moreover, 
$\mathbf{k}_{\lambda}(\Delta(\lambda))\neq 0$.
Since $\lambda(\overline{h})=0$, we have
\begin{displaymath}
e\overline{f}v_{\lambda}=\overline{e}\overline{f}v_{\lambda}=0
\quad\text{ and }\quad
h\overline{f}v_{\lambda}=(\lambda-\alpha)(h)\overline{f}v_{\lambda}.
\end{displaymath}
Therefore, by adjunction, mapping $v_{\lambda-\alpha}$ to 
$\overline{f}v_{\lambda}$, extends to a non-zero homomorphism
from $\Delta(\lambda-\alpha)$ to $\Delta(\lambda)$, implying that
$\mathbf{k}_{\lambda-\alpha}(\Delta(\lambda))\neq 0$.
The claim follows.
\end{proof}

\section{Structure of non-simple Verma modules}\label{s6}

\subsection{Easy case}\label{s6.1}

\begin{proposition}\label{prop51}
Assume that $\lambda\in\overline{\mathfrak{h}}^*$ is such that 
$\lambda(\overline{h})=0$ and $\lambda(h)\notin\mathbb{Z}_{\geq 0}$.
Then there is a short exact sequence
\begin{displaymath}
0\to \Delta(\lambda-\alpha)\to\Delta(\lambda)\to L(\lambda)\to 0. 
\end{displaymath}
\end{proposition}

\begin{proof}
Let $v_{\lambda}$ be the canonical generator of $\Delta(\lambda)$.
From $\lambda(\overline{h})=0$, it follows that 
$e\overline{f}v_{\lambda}=\overline{e}\overline{f}v_{\lambda}=
\overline{h}\overline{f}v_{\lambda}=0$ and
$h\overline{f}v_{\lambda}=(\lambda-\alpha)(h)v_{\lambda}$.
Hence, by adjunction, there is a non-zero homomorphism
$\Delta(\lambda-\alpha)\to\Delta(\lambda)$ sending
$v_{\lambda-\alpha}$ to $\overline{f}v_{\lambda}$.
By the PBW Theorem, this homomorphism is injective and
the quotient $\Delta(\lambda)/\Delta(\lambda-\alpha)$ has a 
basis of the form $\{f^{i}v_{\lambda}\,:\,i\in \mathbb{Z}_{\geq 0}\}$.

Up to a positive integer, $e^{i}f^{i}v_{\lambda}$ is a multiple of
$v_{\lambda}$ with the coefficient $\displaystyle\prod_{j=0}^{i-1}(\lambda(h)-j)$.
As $\lambda(h)\notin\mathbb{Z}_{\geq 0}$, we obtain that 
the quotient $\Delta(\lambda)/\Delta(\lambda-\alpha)$ is a simple
module and hence is isomorphic to $L(\lambda)$. The claim follows.
\end{proof}

\begin{corollary}\label{cor52}
Assume that $\lambda\in\overline{\mathfrak{h}}^*$ is such that 
$\lambda(\overline{h})=0$ and $\lambda(h)\notin\mathbb{Z}_{\geq 0}$.
\begin{enumerate}[$($a$)$]
\item\label{cor52.1}
Then there is a filtration
\begin{displaymath}
\dots\subset \Delta(\lambda-2\alpha)\subset \Delta(\lambda-\alpha)\subset \Delta(\lambda). 
\end{displaymath}
Moreover, all subquotients in this filtration are simple modules and we also have
$\displaystyle \bigcap_{i\in\mathbb{Z}_{\geq 0}}\Delta(\lambda-i\alpha)=0$.
\item\label{cor52.2}  The filtration given by \eqref{cor52.1} is the 
unique composition series of $\Delta(\lambda)$, in other words,
$\Delta(\lambda)$ is a uniserial module.
\end{enumerate}
\end{corollary}

Note that, under the assumptions of Corollary~\ref{cor52},
the module $\Delta(\lambda)$ has infinite length. 
This emphasizes the difference with the classical 
$\mathfrak{sl}_2$-situation, see Subsection~3.2 in \cite{Ma} for the latter.

\begin{proof}
Existence of such filtration and the claim that 
all subquotients in this filtration are simple follows
directly from Proposition~\ref{prop51}.
The claim that $\displaystyle \bigcap_{i\in\mathbb{Z}_{\geq 0}}\Delta(\lambda-i\alpha)=0$
follows from the fact that
$\displaystyle \bigcap_{i\in\mathbb{Z}_{\geq 0}}\mathrm{supp}(\Delta(\lambda-i\alpha))
=\varnothing$, which, in turn, is a consequence of \eqref{eq5}.
This proves claim~\eqref{cor52.1}.

To prove claim~\eqref{cor52.2} we only need to show that any non-zero submodule $M$ of 
$\Delta(\lambda)$ has the form $\Delta(\lambda-i\alpha)$, for some $i$. Choose
$i$ such that $\lambda-i\alpha$ is the highest weight of $M$. From \eqref{cor52.1} 
it follows that any simple subquotient of $\Delta(\lambda)/\Delta(\lambda-i\alpha)$
has a weight of the form $\lambda-j\alpha$, where $j<i$. Therefore 
$M\subset \Delta(\lambda-i\alpha)$. That $M=\Delta(\lambda-i\alpha)$ follows from the
fact that $\Delta(\lambda-i\alpha)$ is generated by its highest weight vector.
This proves claim~\eqref{cor52.2} and completes the proof of the corollary.
\end{proof}

\subsection{Difficult case}\label{s6.2}

\begin{lemma}\label{lem57}
Assume that $\lambda\in\overline{\mathfrak{h}}^*$ is such that 
$\lambda(\overline{h})=0$ and $\lambda(h)=n\in\mathbb{Z}_{\geq 0}$.
Then there are  short exact sequences
\begin{equation}\label{eq57-1}
0\to \Delta(\lambda-\alpha)\to\Delta(\lambda)\to M\to 0
\end{equation}
and
\begin{equation}\label{eq57-2}
0\to L(\lambda-(n+1)\alpha)\to M\to L(\lambda)\to 0.
\end{equation}
\end{lemma}

\begin{proof}
Similarly to Proposition~\ref{prop51}, the vector
$\overline{f}v_{\lambda}$ generates a submodule of
$\Delta(\lambda)$ isomorphic to $\Delta(\lambda-\alpha)$,
giving the exact sequence \eqref{eq57-1}, with
$M=\Delta(\lambda)/\Delta(\lambda-\alpha)$.
The module $M$ is isomorphic to a Verma module for 
$\mathfrak{sl}_2$ and has simple subquotients as
described in \eqref{eq57-2}, see Theorem~3.16 in \cite{Ma}.
\end{proof}

\begin{lemma}\label{lem58}
Assume that $\lambda\in\overline{\mathfrak{h}}^*$ is such that 
$\lambda(\overline{h})=0$ and $\lambda(h)=n\in\mathbb{Z}_{\geq 0}$.
Then the element ${f}^{n+1}v_{\lambda}$ generates
a submodule $K(\lambda)$ of $\Delta(\lambda)$ such that the module
$M_n:=\Delta(\lambda)/K(\lambda)$ is uniserial and has a filtration
\begin{equation}\label{eq58-1}
0=X_k\subset\dots \subset X_1\subset X_{0}=M_n,
\end{equation}
where $k=\lceil\frac{n+1}{2}\rceil$ and
$X_i/X_{i+1}\cong L(\lambda-i\alpha)$, for $i=0,1,\dots,k-1$.
\end{lemma}

\begin{proof}
We prove this statement by induction on $n$.
The induction step moves $\lambda-\alpha$ to $\lambda$
and hence changes $n-2$ to $n$. Therefore we have
two different cases for the basis of the induction.

{\bf Case~1: $n=0$.} In this case 
$efv_{\lambda}=\overline{e}fv_{\lambda}=0$
and $\overline{h}fv_{\lambda}=\overline{f}v_{\lambda}$.
From Lemma~\ref{lem57} we thus get $M_0\cong L(\lambda)$.

{\bf Case~2: $n=1$.} In this case 
$ef^2v_{\lambda}=0$ and 
$\overline{e}f^2v_{\lambda}=2\overline{f}v_{\lambda}$.
Again, from Lemma~\ref{lem57}, we thus get $M_1\cong L(\lambda)$.

Let us now prove the induction step. Consider 
$\Delta(\lambda-\alpha)$ as a submodule of $\Delta(\lambda)$ generated by 
$\overline{f}v_{\lambda}$. We claim that 
$K(\lambda)\cap\Delta(\lambda-\alpha)=K(\lambda-\alpha)$.
Indeed, we have  
$\overline{e}f^{n+1}v_{\lambda}= -n(n+1)f^{n-1}\overline{f}v_{\lambda}$,
and thus $K(\lambda-\alpha)\subset K(\lambda)\cap\Delta(\lambda-\alpha)$.
To prove the reverse inclusion, let us analyze the result of applying a
monomial  $\overline{f}^xf^y\overline{h}^a\overline{e}^pe^qh^b\in U(\mathfrak{g})$ 
to $f^{n+1}v_{\lambda}$. As $f^{n+1}v_{\lambda}$ is a weight element, 
setting $b=0$ changes the outcome by a scalar. As $e\cdot f^{n+1}v_{\lambda}=0$
by the $\mathfrak{sl}_2$-theory, we may assume $q=0$. 
If $x=a=p=0$, then the elements $f^y\cdot f^{n+1}v_{\lambda}$ are linearly
independent and do not belong to $\Delta(\lambda-\alpha)$.
Therefore, if a linear combination of elements of the form
\begin{displaymath}
\overline{f}^xf^y\overline{h}^a\overline{e}^p \cdot f^{n+1}v_{\lambda}
\end{displaymath}
is in $\Delta(\lambda-\alpha)$, then  at least one of $y$, $a$ or $q$ must be non-zero.
Commuting the corresponding overlined basis element to the right and using 
$\overline{e}v_{\lambda}= \overline{h}v_{\lambda}=0$, one shows 
that our linear combination ends up in $K(\lambda-\alpha)$.

Consider now the following diagram:
\begin{equation}\label{eqn1287}
\xymatrix{ 
M_{n-2}\ar@{^{(}-->}[rr]&& M_n\ar@{.>>}[rr]&& \mathrm{Coker}_{M}\\ 
\Delta(\lambda-\alpha)\ar@{^{(}->}[rr]\ar@{-->>}[u]&&
\Delta(\lambda)\ar@{-->>}[u]\ar@{.>>}[rr] &&\mathrm{Coker}_{\Delta}\ar@{.>>}[u]  \\ 
K(\lambda-\alpha)\ar@{^{(}->}[rr]\ar@{^{(}->}[u]&&
K(\lambda)\ar@{^{(}->}[u]\ar@{.>>}[rr] &&\mathrm{Coker}_K \ar@{^{(}.>}[u] 
} 
\end{equation}
The solid part of this diagram consists of natural inclusions.
By the previous paragraph, This solid part is, in fact, a commutative 
pullback diagram. The vertical dashed arrows are natural projections. 
The horizontal dashed arrow is induced by the solid part such that the
dashed box commutes. The horizontal dashed arrow 
is injective since the solid part is a pullback. 
The dotted part of the diagram is given by the Snake Lemma
and the whole diagram \eqref{eqn1287} commutes.
By the Snake Lemma, all rows and all columns of \eqref{eqn1287}
are short exact sequences.

From the Second Isomorphism Theorem and definitions, we have 
\begin{displaymath}
\mathrm{Coker}_{M}\cong \Delta(\lambda)/(K(\lambda)+\Delta(\lambda-\alpha)) 
\cong  L(\lambda). 
\end{displaymath}
Therefore, the upper row of \eqref{eqn1287} gives a short exact sequence
\begin{displaymath}
0\to M_{n-2}\to M_n\to L(\lambda)\to 0. 
\end{displaymath}
As $L(\lambda)$ is a unique simple top of
$\Delta(\lambda)$, the module $L(\lambda)$
also must be a unique simple top of $M_n$.
Now all necessary claims follow by induction.
\end{proof}

As an immediate consequence of the above, we obtain:

\begin{corollary}\label{cor59}
Assume that $\lambda\in\overline{\mathfrak{h}}^*$ is such that 
$\lambda(\overline{h})=0$ and $\lambda(h)=n\in\mathbb{Z}_{\geq 0}$.
The Hasse diagram of the partially ordered, by inclusion, set of 
submodules of $\Delta(\lambda)$ of the form 
$\Delta(\lambda-i\alpha)$ and $K_i$ is as follows
(here $k=\lceil\frac{n-1}{2}\rceil$):
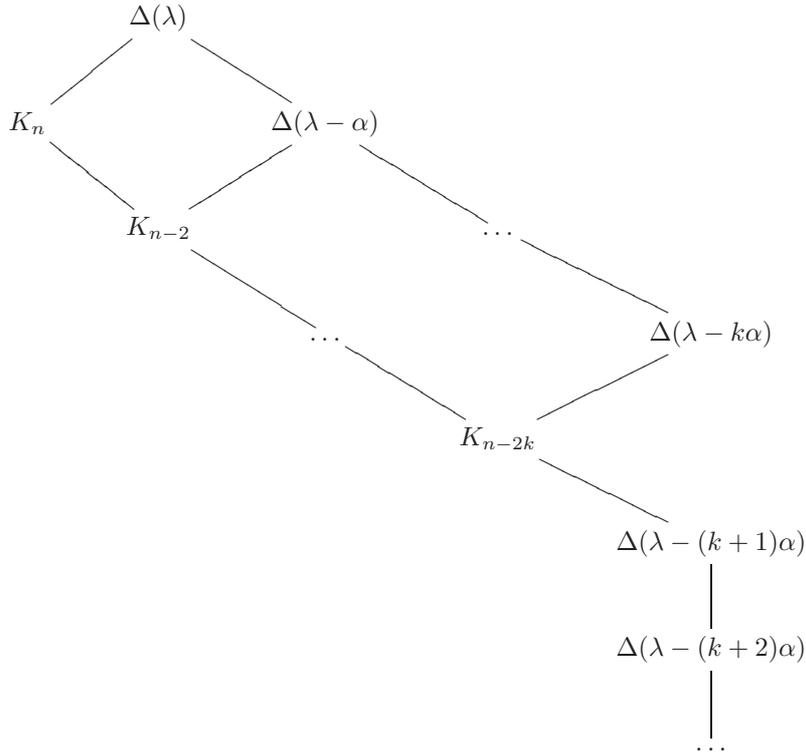
\begin{displaymath}
\xymatrix{ 
&\Delta(\lambda)\ar@{-}[dl]\ar@{-}[dr]&&&\\
K_n\ar@{-}[dr]&&\Delta(\lambda-\alpha)\ar@{-}[dl]\ar@{-}[dr]&&\\
&K_{n-2}\ar@{-}[dr]&&\dots\ar@{-}[dr]&\\
&&\dots\ar@{-}[dr]&&\Delta(\lambda-k\alpha)\ar@{-}[dl]\\
&&&K_{n-2k}\ar@{-}[dr]&\\
&&&&\Delta(\lambda-(k+1)\alpha)\ar@{-}[d]\\
&&&&\Delta(\lambda-(k+2)\alpha)\ar@{-}[d]\\
&&&&\dots
}
\end{displaymath}
\end{corollary}

\section{Gabriel quivers for all blocks}\label{s5}

\subsection{Easy blocks}\label{s5.1}

\begin{theorem}\label{thm15}
For $\lambda\in\overline{\mathfrak{h}}^*_1$, we have:
\begin{enumerate}[$($a$)$]
\item\label{thm15.1} The block $\mathcal{O}(\lambda)$ is equivalent
to the category of finite dimensional $\mathbb{C}[[x]]$-modules.
\item\label{thm15.2} The block $\widetilde{\mathcal{O}}(\lambda)$ is equivalent
to the category of finite dimensional $\mathbb{C}[[x,y]]$-modules.
\end{enumerate}
\end{theorem}

\begin{proof}
Set $x:=\overline{h}-\lambda(\overline{h})$. Then, for any
$M\in \mathcal{O}(\lambda)$, the finite dimensional vector space 
$M^{\lambda}$ is naturally a $\mathbb{C}[[x]]$-module. Moreover, the
functor $F$ sending $M$ to $M^{\lambda}$ and the parabolic induction
functor $G$ are, by the usual hom-tensor adjunction, a pair of adjoint
functors between $\mathcal{O}(\lambda)$ and
the category of finite dimensional $\mathbb{C}[[x]]$-modules.
From the definitions, it follows immediately that they are 
each others quasi inverses, proving claim~\eqref{thm15.1}.

Claim~\eqref{thm15.2} is proved similarly, with 
$x:=\overline{h}-\lambda(\overline{h})$ and
$y:={h}-\lambda({h})$. 
\end{proof}

Recall that the {\em Gabriel quiver} of a block is a directed graph whose 
\begin{itemize}
\item vertices are isomorphism classes of simple objects in the block;
\item the number of arrows from a vertex $L$ to a vertex $S$ equals
the dimension of $\mathrm{Ext}^1(L,S)$.
\end{itemize}
As an immediate consequence of Theorem~\ref{thm15}, we obtain:

\begin{corollary}\label{cor16}
For $\lambda\in\overline{\mathfrak{h}}^*_1$, we have:
\begin{enumerate}[$($a$)$]
\item\label{cor15.1} The Gabriel quiver of $\mathcal{O}(\lambda)$ is:\hspace{2mm} 
$\xymatrix{\bullet\ar@(ur,dr){}}$
\item\label{cor15.2} The  Gabriel quiver of $\widetilde{\mathcal{O}}(\lambda)$ is: 
\hspace{5mm} $\xymatrix{\bullet\ar@(ur,dr){}\ar@(ul,dl){}}$
\end{enumerate}
\end{corollary}

\subsection{Partial simple preserving duality}\label{s5.15}

Denote by $\sigma$ the anti-involution of $\mathfrak{g}$ swapping
$e$ with $f$, and $\overline{e}$ with $\overline{f}$. Note that
$\sigma(h)=h$. 

Let $\mathcal{X}\in\{\widetilde{\mathcal{O}},\mathcal{O}\}$.
Denote by $\mathcal{X}_{\mathrm{fl}}$ the full subcategory of 
$\mathcal{X}$ consisting of modules of finite length. 

For $M\in \mathcal{X}_{\mathrm{fl}}$, we can define on
\begin{displaymath}
M^{\star}:=\bigoplus_{\lambda\in \overline{\mathfrak{h}}^*}
\mathrm{Hom}_{\mathbb{C}}(M^{\lambda},\mathbb{C})
\end{displaymath}
the structure of a $\mathfrak{g}$-module via 
$(a\cdot f)(m):=f(\sigma(a)m)$. Then $M\mapsto M^{\star}$
is a contravariant and involutive self-equivalence of
$\mathcal{X}_{\mathrm{fl}}$. From $\sigma(h)=h$, it follows
that $\mathbf{ch}(M)=\mathbf{ch}(M^{\star})$. In particular,
as simple modules in $\mathcal{X}$ are uniquely determined by their 
characters, it follows that $L(\lambda)^{\star}\cong L(\lambda)$,
for all $\lambda\in \overline{\mathfrak{h}}^*$. In other words,
the {\em duality} $\star$ is simple preserving.

\begin{corollary}\label{cor19}
For all  $\mathcal{X}\in\{\widetilde{\mathcal{O}},\mathcal{O}\}$
and $\lambda,\mu\in \overline{\mathfrak{h}}^*$, we have
\begin{displaymath}
\mathrm{Ext}_{\mathcal{X}}^1(L(\lambda),L(\mu))\cong
\mathrm{Ext}_{\mathcal{X}}^1(L(\mu),L(\lambda)).
\end{displaymath}
\end{corollary}

\begin{proof}
The left hand side of the equality is obtained from the right hand
side by applying the simple preserving duality $\star$.
\end{proof}

We note that $\star$ does not extend to the whole of $\mathcal{X}$
as $\star$ messes up the property of being finitely generated.
For example, for an infinite length Verma module $\Delta(\lambda)\in \mathcal{X}$
as in Subsection~\ref{s6.1}, the module $\Delta(\lambda)^{\star}$
is not finitely generated and hence does not belong to $\mathcal{X}$.

\subsection{Difficult non-integral blocks}\label{s5.2}

\begin{proposition}\label{prop22}
Assume that $\lambda\in\overline{\mathfrak{h}}^*$ is such that 
$\lambda(\overline{h})=0$ and $\lambda(h)\notin\mathbb{Z}$. Then,
for $\mu\in\lambda+\mathbb{Z}\alpha$, we have
\begin{displaymath}
\mathrm{Ext}_{\mathcal{O}}^1(L(\lambda),L(\mu))\cong
\begin{cases}
\mathbb{C}, &  \text{if }\mu=\lambda;\\
\mathbb{C},& \text{if }\mu=\lambda\pm\alpha;\\ 
0,& \text{otherwise}. 
\end{cases}
\end{displaymath}
\end{proposition}

\begin{proof}
By Corollary~\ref{cor19}, without loss of generality we may assume that 
$\mu=\lambda-k\alpha$, for some $k\in\mathbb{Z}_{\geq 0}$.
Let 
\begin{equation}\label{eq22-1}
0\to L(\mu)\to M\to L(\lambda)\to 0 
\end{equation}
be a short exact sequence in $\mathcal{O}$.

Assume first that $k>0$ and that \eqref{eq22-1} does not split. In this case 
$M$ must be generated by $M^{ \lambda}$ and hence, by adjunction, is a quotient 
of the Verma module $\Delta(\lambda)$. Under the assumptions
$\lambda(\overline{h})=0$ and $\lambda(h)\notin\mathbb{Z}$, all submodules
of $\Delta(\lambda)$ are described in Corollary~\ref{cor52}. Out of all
possible quotients of $\Delta(\lambda)$, only  the
quotient $\Delta(\lambda)/\Delta(\lambda-2\alpha)$ has length two.
This quotient has composition subquotients $L(\lambda)$ and
$L(\lambda-\alpha)$. This implies that 
\begin{displaymath}
\mathrm{Ext}_{\mathcal{O}}^1(L(\lambda),L(\lambda-\alpha))\cong\mathbb{C}
\qquad\text{ and }\qquad
\mathrm{Ext}_{\mathcal{O}}^1(L(\lambda),L(\lambda-k\alpha))=0,
\text{ for }k>1.
\end{displaymath}

It remains to compute $\mathrm{Ext}_{\mathcal{O}}^1(L(\lambda),L(\lambda))$.
Consider a non-split short exact sequence \eqref{eq22-1}
in $\mathcal{O}$, with $\lambda=\mu$. The vector space 
$M^{\lambda}$ is, naturally, a $U(\overline{\mathfrak{h}})$-module.
If this module were semi-simple, by adjunction there would exist
two linearly independent homomorphisms from $\Delta(\lambda)$
to $M$ and hence \eqref{eq22-1} would be split. Therefore
$M^{\lambda}$ must be an indecomposable 
$U(\overline{\mathfrak{h}})$-module. As $h$ is supposed to act 
diagonalizably,
such module $M^{\lambda}$ is unique, up to isomorphism.
In particular, there is a basis $\{v,w\}$ of $M^{\lambda}$ such that
the matrix of the action of $\overline{h}$ in this basis is
\begin{displaymath}
\left(\begin{array}{cc}0&0\\1&0\end{array}\right).
\end{displaymath}
Consider now the module $\displaystyle
\Delta(M^{\lambda}):=U(\mathfrak{g})
\bigotimes_{U(\overline{\mathfrak{b}})}M^{\lambda}$,
where  $\overline{\mathfrak{n}}_+M^{\lambda}=0$.
By adjunction, $\Delta(M^{\lambda})$ surjects onto
$M$. Hence, we just need to check how many 
submodules $K$ of $\Delta(M^{\lambda})$ have the
property that $\Delta(M^{\lambda})/K$ has length 
two with both composition subquotients isomorphic to 
$L(\lambda)$. We claim that such submodule is unique,
which implies that 
$\mathrm{Ext}_{\mathcal{O}}^1(L(\lambda),L(\lambda))\cong\mathbb{C}$.
In fact, since  $\mathbf{k}_{\lambda}(\Delta(M^{\lambda}))=2$
by construction, the uniqueness of $K$, provided that $K$ exists, is clear.

To prove existence, we consider the submodule $K$ of 
$\Delta(M^{\lambda})$ generated by $\overline{f}w$ and
$\lambda(h)\overline{f}v-fw$ (note that $\lambda(h)\neq 0$ 
by our assumptions). It is easy to check that both these vectors
are annihilated by $e$ and $\overline{e}$. The vector $\overline{f}w$
generates a submodule of $\Delta(M^{\lambda})$ isomorphic to 
$\Delta(\lambda-\alpha)$. The image of $\lambda(h)\overline{f}v-fw$
in the quotient $\Delta(M^{\lambda})/\Delta(\lambda)$ generates
in this quotient a submodule isomorphic to $\Delta(\lambda-\alpha)$.
Therefore, from Proposition~\ref{prop51} it follows that 
$\Delta(M^{\lambda})/K$ indeed has length two with both simple
subquotients isomorphic to $L(\lambda)$. The claim follows.
\end{proof}

As an immediate corollary from Proposition~\ref{prop22}, we have:

\begin{corollary}\label{cor23}
Assume that $\lambda\in\overline{\mathfrak{h}}^*$ is such that 
$\lambda(\overline{h})=0$ and $\lambda(h)\notin\mathbb{Z}$. Then,
for $\xi:=\lambda+\mathbb{Z}\alpha$, the Gabriel quiver of 
$\mathcal{O}(\xi)$ has the form:
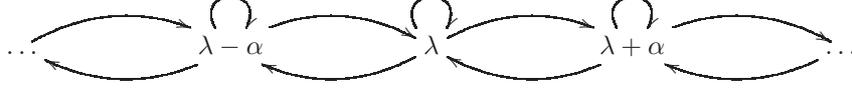
\vspace{1mm}

\begin{displaymath}
\xymatrix{
\dots\ar@/^1pc/[rr]&&
\lambda-\alpha \ar@(ul,ur)[]\ar@/^1pc/[rr]\ar@/^1pc/[ll]&& 
\lambda \ar@(ul,ur)[]\ar@/^1pc/[rr]\ar@/^1pc/[ll]&&
\lambda+\alpha \ar@(ul,ur)[]\ar@/^1pc/[rr]\ar@/^1pc/[ll]&&
\dots\ar@/^1pc/[ll]
}
\end{displaymath}
\end{corollary}
\vspace{1mm}

Now we can proceed to $\widetilde{\mathcal{O}}$.

\begin{proposition}\label{prop24}
Assume that $\lambda\in\overline{\mathfrak{h}}^*$ is such that 
$\lambda(\overline{h})=0$ and $\lambda(h)\notin\mathbb{Z}$. Then,
for $\mu\in\lambda+\mathbb{Z}\alpha$, we have
\begin{displaymath}
\mathrm{Ext}_{\widetilde{\mathcal{O}}}^1(L(\lambda),L(\mu))\cong
\begin{cases}
\mathbb{C}^2, &  \text{if }\mu=\lambda;\\
\mathbb{C},& \text{if }\mu=\lambda\pm\alpha;\\ 
0,& \text{otherwise}. 
\end{cases}
\end{displaymath}
\end{proposition}

\begin{proof}
The case $\mu\neq\lambda$ is proved by exactly the same 
arguments as in Proposition~\ref{prop22}. The case
$\mu=\lambda$ is also similar, but requires some small adjustments
which we describe below.

Consider a non-split short exact sequence \eqref{eq22-1}
in $\widetilde{\mathcal{O}}$, with $\lambda=\mu$. The vector space 
$M^{\lambda}$ is an indecomposable $U(\overline{\mathfrak{h}})$-module
of length two, namely, a self-extension of the simple
$U(\overline{\mathfrak{h}})$-module $\mathbb{C}_{\lambda}$
corresponding to $\lambda$. The space of such self-extensions 
is two-dimensional 
(as $\overline{\mathfrak{h}}$ is two-dimensional). In fact, 
using the arguments as in the proof of Proposition~\ref{prop22}, 
we can show that parabolic induction, followed by taking a 
canonical quotient, defines a surjective map from
$\mathrm{Ext}^1_{U(\overline{\mathfrak{h}})}
(\mathbb{C}_{\lambda},\mathbb{C}_{\lambda})$ to
$\mathrm{Ext}_{\mathcal{O}}^1(L(\lambda),L(\lambda))$
which sends isomorphic module to isomorphic and non-isomorphic
modules to non-isomorphic (the latter claim is obvious by 
restricting the action to the generalized $\lambda$-weight space). 
This, clearly, implies the necessary
claim. Here are the details.

There is a basis $\{v,w\}$ of $M^{\lambda}$ such that
the matrices of the action of $h$ and $\overline{h}$ in this basis are
\begin{displaymath}
\left(\begin{array}{cc}\lambda(h)&0\\p&\lambda(h)\end{array}\right)
\qquad\text{ and }\qquad
\left(\begin{array}{cc}0&0\\q&0\end{array}\right),
\end{displaymath}
respectively, where $p$ and $q$ are complex numbers at least one of which
is non-zero. Similarly to the proof of Proposition~\ref{prop22},
one shows that the submodule $K$ of $\Delta(M^{\lambda})$
generated by $\overline{f}w$ and
$\frac{\lambda(h)}{q}\overline{f}v-fw$, in case $q\neq 0$,
or $\overline{f}v$, in case $q=0$, is the unique submodule of
$\Delta(M^{\lambda})$ such that $\Delta(M^{\lambda})/K$
is isomorphic to $M$. The claim follows.
\end{proof}

As an immediate corollary from Proposition~\ref{prop24}, we have:

\begin{corollary}\label{cor25}
Assume that $\lambda\in\overline{\mathfrak{h}}^*$ is such that 
$\lambda(\overline{h})=0$ and $\lambda(h)\notin\mathbb{Z}$. Then,
for $\xi:=\lambda+\mathbb{Z}\alpha$, the Gabriel quiver of 
$\widetilde{\mathcal{O}}(\xi)$ has the form:
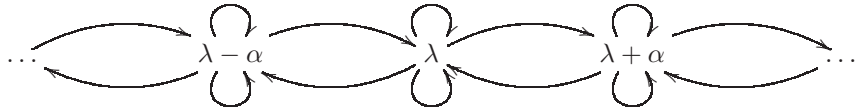
\vspace{1mm}

\begin{displaymath}
\xymatrix{
\dots\ar@/^1pc/[rr]&&
\lambda-\alpha \ar@(ul,ur)[]\ar@(dl,dr)[]\ar@/^1pc/[rr]\ar@/^1pc/[ll]&& 
\lambda \ar@(ul,ur)[]\ar@(dl,dr)[]\ar@/^1pc/[rr]\ar@/^1pc/[ll]&&
\lambda+\alpha \ar@(ul,ur)[]\ar@(dl,dr)[]\ar@/^1pc/[rr]\ar@/^1pc/[ll]&&
\dots\ar@/^1pc/[ll]
}
\end{displaymath}
\end{corollary}
\vspace{1mm}

\subsection{Other self-extensions of simples}\label{s5.3}

\begin{corollary}\label{cor29}
Let $\lambda\in\overline{\mathfrak{h}}^*$ be such that 
$\lambda(\overline{h})=0$ and $\lambda(h)\notin\mathbb{Z}_{\geq 0}$.
Then  we have
\begin{displaymath}
\mathrm{Ext}_{\mathcal{O}}^1(L(\lambda),L(\lambda))\cong\mathbb{C}
\qquad\text{ and }\qquad
\mathrm{Ext}_{\widetilde{\mathcal{O}}}^1(L(\lambda),L(\lambda))
\cong\mathbb{C}^2.
\end{displaymath}
\end{corollary}

\begin{proof}
The follows directly from the corresponding parts in the
proofs of  Proposition~\ref{prop22} and Proposition~\ref{prop24}.
\end{proof}

\begin{lemma}\label{lem27-1}
We have
\begin{displaymath}
\mathrm{Ext}_{\mathcal{O}}^1(L(0),L(0))=
\mathrm{Ext}_{\widetilde{\mathcal{O}}}^1(L(0),L(0))=0.
\end{displaymath}
\end{lemma}

\begin{proof}
The elements $e$, $f$, $\overline{e}$ and $\overline{f}$
must annihilate any self-extension $M$ of $L(0)$ since
$M^{\pm\alpha}=0$, for such $M$.
As $h=[e,f]$ and $\overline{h}=[\overline{e},f]$, it follows
that both $h$ and $\overline{h}$ must annihilate $M$ as well.
Therefore $M$ splits.
\end{proof}

\begin{proposition}\label{prop27}
Let $\lambda\in\overline{\mathfrak{h}}^*$ be such that 
$\lambda(\overline{h})=0$ and $\lambda(h)\in\mathbb{Z}_{>0}$.
Then  we have
\begin{displaymath}
\mathrm{Ext}_{\mathcal{O}}^1(L(\lambda),L(\lambda))\cong
\mathrm{Ext}_{\widetilde{\mathcal{O}}}^1(L(\lambda),L(\lambda))
\cong\mathbb{C}.
\end{displaymath}
\end{proposition}

\begin{proof}
By Weyl's complete reducibility theorem, $h$ acts diagonalizably
on any finite dimensional $\mathfrak{g}$-module.
Hence any self-extension of $L(\lambda)$ lives in 
$\mathcal{O}$. Therefore it is enough to prove that 
$\mathrm{Ext}_{\mathcal{O}}^1(L(\lambda),L(\lambda))\cong\mathbb{C}$.

Let $M$ be a self-extension of $L(\lambda)$. Then
$M^{\lambda}$ is an $\mathbb{C}[\overline{h}]$-module and,
similarly to Proposition~\ref{prop22}, $M$ is indecomposable
if and only if $M^{\lambda}$ is. As
$\mathbb{C}[\overline{h}]$ is a polynomial algebra in one variable,
this implies that 
$\mathrm{Ext}_{\mathcal{O}}^1(L(\lambda),L(\lambda))$ is at most
one dimensional. To prove that 
$\mathrm{Ext}_{\mathcal{O}}^1(L(\lambda),L(\lambda))$
is exactly one-dimensional, it is enough to construct one
non-split self-extension of $L(\lambda)$, which we do below.

Let $n:=\lambda(h)\in\mathbb{Z}_{\geq 0}$. By Example~1.24 in \cite{Ma}, 
$L(\lambda)$ has a basis $\{v_{0},v_1,\dots,v_n\}$ such that
\begin{displaymath}
ev_{i}=iv_{i-1},\quad
fv_{i}=(n-i)v_{i+1},\quad
hv_{i}=(n-2i)v_i,\quad\text{ for }\,\, i=0,1,\dots,n.
\end{displaymath}
Take another copy $\overline{L(\lambda)}$  of $L(\lambda)$ with basis 
$\{\overline{v}_{0},\overline{v}_1,\dots,\overline{v}_n\}$
and similarly defined action. Consider 
$M=L(\lambda)\oplus\overline{L(\lambda)}$ and define
\begin{displaymath}
\overline{e}v_{i}=i\overline{v}_{i-1},\quad
\overline{f}v_{i}=(n-i)\overline{v}_{i+1},\quad
\overline{h}v_{i}=(n-2i)\overline{v}_i,\quad\text{ for }\,\, i=0,1,\dots,n, 
\end{displaymath}
and $\overline{e}\overline{L(\lambda)}=
\overline{f}\overline{L(\lambda)}=\overline{h}\overline{L(\lambda)}=0$.
It is straightforward that this defines on $M$ the structure of
a $\mathfrak{g}$-module. As the action of $\overline{h}$ on $v_0$
is non-zero (here the condition $n>0$ is crucial!), the module
$M$ is a non-split self-extension of $L(\lambda)$.
This completes the proof.
\end{proof}

\subsection{Difficult integral blocks}\label{s5.4}

\begin{proposition}\label{prop32}
Let $\lambda\in\overline{\mathfrak{h}}^*$ be such that
$\lambda(\overline{h})= 0$.
\begin{enumerate}[$($a$)$]
\item\label{prop32.1} We have
\begin{displaymath}
\mathrm{Ext}_{\mathcal{O}}^1(L(\lambda),L(\lambda-\alpha))\cong
\mathrm{Ext}_{\widetilde{\mathcal{O}}}^1(L(\lambda),L(\lambda-\alpha))\cong
\mathbb{C}. 
\end{displaymath}
\item\label{prop32.2} If $\lambda(h)=n\in\mathbb{Z}_{\geq 0}$,
then we have
\begin{displaymath}
\mathrm{Ext}_{\mathcal{O}}^1(L(\lambda),L(\lambda-(n+1)\alpha))\cong
\mathrm{Ext}_{\widetilde{\mathcal{O}}}^1(L(\lambda),L(\lambda-(n+1)\alpha))
\cong\mathbb{C}. 
\end{displaymath}
\item\label{prop32.3} If $\lambda(h)\neq n\in\mathbb{Z}_{>0}$,
then we have
\begin{displaymath}
\mathrm{Ext}_{\mathcal{O}}^1(L(\lambda),L(\lambda-(n+1)\alpha))=
\mathrm{Ext}_{\widetilde{\mathcal{O}}}^1(L(\lambda),L(\lambda-(n+1)\alpha))
=0. 
\end{displaymath}

\end{enumerate}
\end{proposition}

\begin{proof}
We start with claim~\eqref{prop32.1}. Assume that
\begin{displaymath}
0\to L(\lambda-\alpha)\to M\to L(\lambda)\to 0 
\end{displaymath}
is a non-split short exact sequence. Then, similarly to 
Proposition~\ref{prop22}, $M$ must be a quotient of
$\Delta(\lambda)$. If  $\lambda(h)\notin\mathbb{Z}_{\geq 0}$,
then from Corollary~\ref{cor52} it follows that
$\Delta(\lambda)$  has a unique quotient with 
correct composition subquotients. 
If  $\lambda(h)\in\mathbb{Z}_{\geq 0}$,
then from Lemma~\ref{lem58} it follows that
$\Delta(\lambda)$  has a unique quotient with 
correct composition subquotients. This completes the proof of
claim~\eqref{prop32.1}

We proceed with claim~\eqref{prop32.2}. Assume that
$\lambda(h)=n\in\mathbb{Z}_{\geq 0}$ and
\begin{displaymath}
0\to L(\lambda-(n+1)\alpha)\to M\to L(\lambda)\to 0 
\end{displaymath}
is a non-split short exact sequence. Then, from 
Lemma~\ref{lem57} it follows that
$\Delta(\lambda)$  has a unique quotient with 
correct composition subquotients. This completes the proof of
claim~\eqref{prop32.2}.

Proposition~\ref{prop51}, Lemma~\ref{lem57} and Lemma~\ref{lem58}
imply that the only socle components possible in length
two quotients of $\Delta(\lambda)$ are $\Delta(\lambda-\alpha)$
and $\Delta(\lambda-(n+1)\alpha)$, and the latter one is
only possible under the additional
assumption that $\lambda(h)=n\in\mathbb{Z}_{\geq 0}$.
This implies claim~\eqref{prop32.3} and completes the proof. 
\end{proof}

Combining Proposition~\ref{prop32}, 
Corollary~\ref{cor29}, Lemma~\ref{lem27-1}, Corollary~\ref{cor19}
and Proposition~\ref{prop27}, we obtain:

\begin{corollary}\label{cor33}
{\hspace{1mm}}

\begin{enumerate}[$($a$)$]
\item \label{cor33.1}
The Gabriel quiver of 
${\mathcal{O}}(\mathbb{Z}\alpha)$ is:
\vspace{1mm}

\begin{displaymath}
\xymatrix{
0\ar@/^1pc/[rr]\ar@/^1pc/[d]&& 
\alpha\ar@(ul,ur)[]\ar@/^1pc/[rr]\ar@/^1pc/[ll]\ar@/^1pc/[d]&& 
2\alpha\ar@(ul,ur)[]\ar@/^1pc/[rr]\ar@/^1pc/[ll]\ar@/^1pc/[d]&& 
\dots\ar@/^1pc/[ll]\\
\text{-}\alpha\ar@(dr,dl)[]\ar@/^1pc/[rr]\ar@/^1pc/[u]&& 
\text{-}2\alpha\ar@(dr,dl)[]\ar@/^1pc/[rr]\ar@/^1pc/[ll]\ar@/^1pc/[u]&& 
\text{-}3\alpha\ar@(dr,dl)[]\ar@/^1pc/[rr]\ar@/^1pc/[ll]\ar@/^1pc/[u]&& 
\dots\ar@/^1pc/[ll]\\
}
\end{displaymath}
\vspace{5mm}
\item\label{cor33.2}
The Gabriel quiver of 
$\widetilde{\mathcal{O}}(\mathbb{Z}\alpha)$ is:
\vspace{1mm}

\begin{displaymath}
\xymatrix{
0\ar@/^1pc/[rr]\ar@/^1pc/[d]&& 
\alpha\ar@(ul,ur)[]\ar@/^1pc/[rr]\ar@/^1pc/[ll]\ar@/^1pc/[d]&& 
2\alpha\ar@(ul,ur)[]\ar@/^1pc/[rr]\ar@/^1pc/[ll]\ar@/^1pc/[d]&& 
\dots\ar@/^1pc/[ll]\\
\text{-}\alpha\ar@(d,dr)[]\ar@(d,dl)[]
\ar@/^1pc/[rr]\ar@/^1pc/[u]&& 
\text{-}2\alpha\ar@(d,dr)[]\ar@(d,dl)[]
\ar@/^1pc/[rr]\ar@/^1pc/[ll]\ar@/^1pc/[u]&& 
\text{-}3\alpha\ar@(d,dr)[]\ar@(d,dl)[]
\ar@/^1pc/[rr]\ar@/^1pc/[ll]\ar@/^1pc/[u]&& 
\dots\ar@/^1pc/[ll]\\
}
\end{displaymath}
\vspace{5mm}
\item\label{cor33.3}
The Gabriel quiver of 
${\mathcal{O}}(\frac{1}{2}\alpha+\mathbb{Z}\alpha)$ is:
\vspace{1mm}

\begin{displaymath}
\xymatrix{
&&
\frac{1}{2}\alpha\ar@/^1pc/[rr]\ar@/^1pc/[d]\ar@/_2pc/[dll]\ar@(ul,ur)[]&& 
\frac{1}{2}\alpha\text{$+$}\alpha\ar@(ul,ur)[]\ar@/^1pc/[rr]\ar@/^1pc/[ll]\ar@/^1pc/[d]&& 
\frac{1}{2}\alpha\text{$+$}2\alpha\ar@(ul,ur)[]\ar@/^1pc/[rr]\ar@/^1pc/[ll]\ar@/^1pc/[d]&& 
\dots\ar@/^1pc/[ll]\\
\frac{1}{2}\alpha\text{-}\alpha\ar@/^1pc/[rr]\ar@(dr,dl)[]\ar@/^1pc/[rru]&&
\frac{1}{2}\alpha\text{-}2\alpha
\ar@(dr,dl)[]\ar@/^1pc/[rr]\ar@/^1pc/[u]\ar@/^1pc/[ll]&& 
\frac{1}{2}\alpha\text{-}3\alpha
\ar@(dr,dl)[]\ar@/^1pc/[rr]\ar@/^1pc/[ll]\ar@/^1pc/[u]&& 
\frac{1}{2}\alpha\text{-}4\alpha
\ar@(dr,dl)[]\ar@/^1pc/[rr]\ar@/^1pc/[ll]\ar@/^1pc/[u]&& 
\dots\ar@/^1pc/[ll]\\
}
\end{displaymath}
\vspace{5mm}
\item\label{cor36}
The Gabriel quiver of 
$\widetilde{\mathcal{O}}(\frac{1}{2}\alpha+\mathbb{Z}\alpha)$ is:
\vspace{1mm}

\begin{displaymath}
\xymatrix{
&&
\frac{1}{2}\alpha\ar@/^1pc/[rr]\ar@/^1pc/[d]\ar@/_2pc/[dll]\ar@(ul,ur)[]&& 
\frac{1}{2}\alpha\text{$+$}\alpha\ar@(ul,ur)[]\ar@/^1pc/[rr]\ar@/^1pc/[ll]\ar@/^1pc/[d]&& 
\frac{1}{2}\alpha\text{$+$}2\alpha\ar@(ul,ur)[]\ar@/^1pc/[rr]\ar@/^1pc/[ll]\ar@/^1pc/[d]&& 
\dots\ar@/^1pc/[ll]\\
\frac{1}{2}\alpha\text{-}\alpha\ar@/^1pc/[rr]
\ar@(d,dr)[]\ar@(d,dl)[]\ar@/^1pc/[rru]&&
\frac{1}{2}\alpha\text{-}2\alpha
\ar@(d,dr)[]\ar@(d,dl)[]\ar@/^1pc/[rr]\ar@/^1pc/[u]\ar@/^1pc/[ll]&& 
\frac{1}{2}\alpha\text{-}3\alpha
\ar@(d,dr)[]\ar@(d,dl)[]\ar@/^1pc/[rr]\ar@/^1pc/[ll]\ar@/^1pc/[u]&& 
\frac{1}{2}\alpha\text{-}4\alpha
\ar@(d,dr)[]\ar@(d,dl)[]\ar@/^1pc/[rr]\ar@/^1pc/[ll]\ar@/^1pc/[u]&& 
\dots\ar@/^1pc/[ll]\\
}
\end{displaymath}
\end{enumerate}
\end{corollary}
\vspace{7mm}

\textbf{Acknowledgements:} This research was partially supported by
the Swedish Research Council and G{\"o}ran Gustafsson Stiftelse.
We thank the referee for helpful comments.

\vspace{2mm}

\noindent
V.~M.: Department of Mathematics, Uppsala University, Box. 480,
SE-75106, Uppsala,\\ SWEDEN, email: {\tt mazor\symbol{64}math.uu.se}

\noindent
C.~S.: Department of Mathematics, Uppsala University, Box. 480,
SE-75106, Uppsala,\\ SWEDEN, email: {\tt christoffer.soderberg\symbol{64}math.uu.se}

\end{document}